\documentclass[10pt]{amsart}
\usepackage{amssymb, amscd, amsmath, amsthm, latexsym, hyperref,xcolor}
\usepackage{pb-diagram, fancyhdr, graphicx, psfrag}
 \usepackage[text={6.0in,8.5in},centering,letterpaper,dvips]{geometry}
\newcommand{\compactlist}{\begin{list}{$\bullet$}{\setlength{\leftmargin}{1em}}}

\def\cs{\mathop{\#}}

\def\calf{\mathcal{F}}

\def\cala{\mathcal{A}}

\def\calt{\mathcal{T}}
\def\cala{\mathcal{A}}
\def\cfk{{\textrm{CFK}}}

\newcommand{\spinc}{\ifmmode{{\mathfrak s}}\else{${\mathfrak s}$\ }\fi}
\newcommand{\spinct}{\ifmmode{{\mathfrak t}}\else{${\mathfrak t}$\ }\fi}
\newcommand{\spincw}{\ifmmode{{\mathfrak w}}\else{${\mathfrak w}$\ }\fi}
\def\Z{\mathbb Z}

\def\R{\mathbb R}

\def\F{\mathbb F}

\def\crm{\mathrm{CF}}
\def\crmm{\mathrm{C}}
\newcommand{\fig}[2] { \includegraphics[scale=#1]{#2} }

\def\ck{\mathrm{CF}(K)}
\def\ct{\mathrm{CF}_t}
\def\L{\Lambda}
\newtheorem{theorem}{Theorem}[section]

\newtheorem*{theorem*}{Theorem}
\newtheorem{lemma}[theorem]{Lemma}
\newtheorem{corollary}[theorem]{Corollary}

\newtheorem{proposition}[theorem]{Proposition}
\theoremstyle{definition}
\newtheorem{definition}[theorem]{Definition}

\theoremstyle{remark}

\numberwithin{equation}{section}

\def\U{\Upsilon}

\begin{document}

%%%%%%%HEADER%%%%%%%%%%

\title[Upsilon Notes]{Notes on the knot concordance invariant Upsilon}
\author{Charles Livingston}
\address{Charles Livingston: Department of Mathematics, Indiana University, Bloomington, IN 47405 }
\email{livingst@indiana.edu}

%\subject{primary}{msc2000}{57M25}
%\keyword{knot concordance, four-genus, Upsilon, Heegaard Floer}

\begin{abstract}  Ozsv\'ath, Stipsicz, and Szab\'o have defined a knot concordance invariant $\U_K$   taking  values in the group of piecewise linear functions on the closed interval $[0,2]$.  This paper presents a description of one approach to defining $\U_K$ and  proving its basic properties.
\end{abstract}

%\begin{asciiabstract}  Ozsvath, Stipsicz, and Szabo have defined a knot concordance invariant Upsilon   taking  values in the group of piecewise linear functions on the closed interval  [0,2].  This paper presents a description of one approach to defining Upsilon and  proving its basic properties.
%\end{asciiabstract}

\thanks{The author was supported by a Simons Foundation grant and by NSF-DMS-1505586.}
%\date{\today}
 
 \maketitle
 %%%%%%%HEADER%%%%%%%%%%

 %%%%%%%%%%%%SECTION%%%%%%%%%%%%%%%%
 \section{Introduction}
 
 In~\cite{oss}, Ozsv\'ath, Stipsicz,  and  Szab\'o used the Heegaard Floer knot complex $\cfk^-(K)$ of a knot $K\subset S^3$   to define a  piecewise linear function $\U_K(t)$ with domain $[0,2]$.  The function $K \to \U_K$ induces a homomorphism  from the smooth knot concordance group  to the group of   functions on the interval $[0,2]$.  Among its properties,  $\U_K(t)$ provides bounds on the four-genus,  $g_4(K)$, the three-genus, $g_3(K)$, and, consequently, the concordance genus, $g_c(K)$.
 This note describes a simple approach to defining   $\U_K(t)$ using $\cfk^\infty(K)$ and proving its basic properties. \vskip.05in
 
 \noindent{\it Acknowledgments} Thanks go to  Jen Hom,  Slaven Jabuka, Swatee Naik, Peter Ozsv\'ath,  Shida Wang, and C.-M. Michael Wong for their comments.  Matt Hedden pointed out  the structure theorem for filtered  knot complexes presented in the appendix  and its usefulness in simplifying a key proof.  Suggestions from the referee led to valuable improvements in the exposition.  %The author was supported by a Simons Foundation grant and by NSF-DMS-1505586.
 
 %%%%%%%%%%%%SECTION%%%%%%%%%%%%%%%%
  \section{Knot Complexes}\label{section:complexes}

We begin by describing the algebraic structure of the  Heegaard Floer  complex  of a knot $K$, denoted  $\cfk^\infty(K)$, first defined in~\cite{os-knotinvars}.  This is a vector space over the field $\F$ with two elements.  To simplify notation, we write $\ck$ for $\cfk^\infty(K)$.  Here we summarize its  basic properties.

\begin{itemize}
\item The chain complex $\ck$ has an integer valued grading and the boundary map $\partial$ is of degree $-1$.   The grading is called the {\it Maslov grading}.  
The grading of a homogeneous element is denoted $gr(x)$.

\item  The complex $\ck$ has an {\it Alexander filtration} consisting of an increasing sequence of subcomplexes.  The filtration level of an element $x \in \ck$ is denoted $Alex(x)$. 

\item  There is a similar filtration, called the {\it algebraic filtration}, and filtration levels of elements are denoted $Alg(x)$. 
\item There is an action of the Laurent polynomial ring $\F[U,U^{-1}]$ on $\ck$.  The action of $U$ commutes with $\partial$, lowers gradings by $2$, and lowers Alexander and algebraic filtration levels by 1.  
 
\item  Let $\L$ denote $\F[U,U^{-1}]$.   As a $\L$--module, $\ck$ is free on a finite set of generators, $\{x_i\}_{1\le i \le r}$. To simplify notation, we suppress the indexing set.  The set of elements $\{ U^k x_i \}_{k \in \Z}$ forms a bilfiltered graded basis for $\ck$:  for any triple of integers, $(g, m,n)$, the subspace of $\ck$ spanned by elements of grading $g$, Alexander filtration level less than or equal to $m$, and algebraic filtration level less than or equal to $n$, has as basis a subset of  $\{U^k x_i\}$. 
 
\item  The  singly filtered complex $(\ck, Alg)$  with $\L$--structure is  chain homotopy equivalent to complex $\calt \cong \Lambda$ where $1 \in \L$ has grading 0 and filtration level 0, and the boundary map is  trivial.    (The same statement holds for the Alexander grading, but we do not use this fact.) 

\end{itemize}

The construction of $\ck$ depends on a series of choices.  However, there is a natural definition of chain homotopy equivalence for graded, bifiltered chain complexes with $\L$--action.  A key result of~\cite{os-knotinvars} is that in this sense, the   chain homotopy equivalence class of $\ck$ is a well-defined knot invariant.

As an example, Figure~\ref{fig37} presents a schematic diagram of the   complex for the torus knot $T(3,7)$.   As a $\L$--module it has nine filtered generators, with algebraic and  Alexander   filtration levels indicated by the first and second coordinate, respectively.  Five of the generators, indicated with black dots,  have grading 0; the four white dots represent generators of  grading one.  The boundary map is   indicated by the arrows.   The rest of $\ck$ is the direct sum  of the $U^k$, $k \in \Z$, translates of this finite complex; for instance, applying $U$ shifts the diagram one down and to the left.

\begin{figure}[b]
\center{\fig{.7}{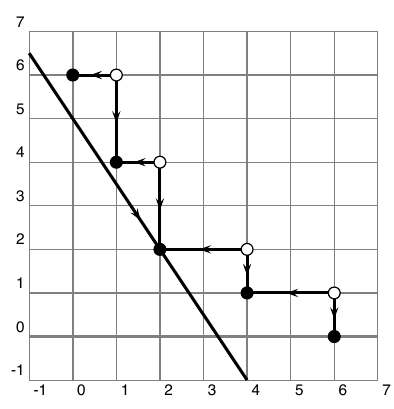}}
\caption{$\cfk^\infty(T(3,7))$}
\label{fig37}
\end{figure}

%%%%%%%%%%%%SECTION%%%%%%%%%%%%%%%%
\section{Filtrations}

We now discuss more general filtrations on vector spaces.  In our applications, the vector space will be $\ck.$

 \begin{definition}
A  {\it real-valued (discrete) filtration} on a vector space $\crmm$ is     a   collection of subspaces $\calf = \{\crmm_s\}$ indexed by the real numbers.  This collection must satisfy the following properties:   
\begin{enumerate}
\item $\crmm_{s_1} \subseteq \crmm_{s_2} $ if $s_1 \le s_2$. 
\item $\crmm = \cup_{s\in \R} \crmm_s$. 
\item $\cap_{s\in \R} \crmm_s = \{0\}$. 
\item  ({\it discreteness}) $\crmm_{s_2}/\crmm_{s_1}$ is finite dimensional when $s_1 \le s_2$.
\end{enumerate}
\end{definition}

Given a discrete filtration $\calf = \{\crmm_s\}$ on $\crmm$, we can define an associated function on $\crmm$, which we temporarily also denote by $\calf$, given by $\calf(x) = \min\{ s \in \R\ |\  x \in \crmm_s\}$.  Notice that $\calf^{-1}((-\infty, s] ) = \crmm_s$.

Given an arbitrary real-valued  function $f$ on $\crmm$, one can define an associated filtration with $\crmm_s = \text{Span}(f^{-1}((-\infty, s]))$.  The resulting filtration need not be discrete.
 
\noindent{\bf Notation.}   In cases in which more than one filtration might be under consideration,  we will write $(\crmm,\calf)_s$ rather than $\crmm_s$.
\vskip.05in 
 
\begin{definition} A set of vectors $\{z_i\}$ in the real filtered vector  space $\crmm$  is called a {\it filtered basis} if it is linearly independent  and every $\crmm_s$ has some subset of $\{z_i\}$ as a basis.   If $\crmm$ is also graded, $\crmm = \oplus_{i=-\infty}^\infty G_i $, then we say the basis is a filtered graded basis if each $\crmm_s \cap G_k$ has a subset of $\{z_i\}$ as a basis. 
\end{definition}

%%%%%%%%%%%%SECTION%%%%%%%%%%%%%%%%
\section{ The definition of  the filtration $\calf_t$ on $\ck$.}
 
For any $t\in [0,2]$, the convex combination of Alexander and algebraic filtrations, $  \frac{t}{2} Alex + (1-\frac{t}{2})Alg$, defines a  real-valued function on $\ck$, to which we associate a filtration denoted $\calf_t$.  That is, for all $s \in \R$, $(\ck, \calf_t)_s$ is spanned by all vectors   $x \in \ck$ such that $\frac{t}{2}Alex(x) + (1 - \frac{t}{2})Alg(x) \le s$.
 
\begin{theorem} If $0\le t \le 2$, the filtration $\calf_t$ on $\ck$ is a filtration by subcomplexes and is discrete.  The action of $U$ lowers filtration levels by 1.
\end{theorem}
\begin{proof} To see that these are subcomplexes, suppose that $x\in (\ck,\calf_t)_s$.  Write $x = \sum x_i$ where $\frac{t}{2}Alex(x_i) + (1 - \frac{t}{2})Alg(x_i) \le s$ for all $i$.  Since $\partial x = \sum \partial x_i$, we only  need  to check that for each $i$,  $\partial x_i \in (\ck,\calf_t)_s$.  Let $x_i$ have $Alex(x_i) = a$ and $Alg(x_i) = b$.  Then $Alex(\partial x_i) = a' \le a$ and $Alg(\partial x_i) = b' \le b.$  Since both $\frac{t}{2}$ and $(1- \frac{t}{2})$ are nonnegative,  $  \frac{t}{2}a' + (1 - \frac{t}{2})b' \le  \frac{t}{2}a + (1 - \frac{t}{2})b \le s$, as desired.
 
The discreteness of the filtration depends on two properties of $\ck$.  First, letting $g$ denote the three-genus, $ g_3(K)$, according to~\cite{os-GenusBounds} one has  $-g \le Alex(x) - Alg(x) \le g$ for all $x$.  From this it follows that for given $s_1 < s_2$, there are $k_1$ and $k_2$ in $\R$ such that 
 $$(\ck, Alex)_{k_1} \subseteq (\ck, \calf_t)_{s_1} \subseteq (\ck, \calf_t)_{s_2} \subseteq (\ck, Alex)_{k_2}.$$  (The values of $k_1$ and $k_2$ can be chosen to be $ s_1 - (1-\frac{t}{2})g$ and  $s_2 + (1-\frac{t}{2})g$, respectively, but we do not need this level of detail.)
Second, the Alexander filtration is discrete, so the quotient ($\ck, Alex)_{k_2} / (\ck, Alex)_{k_1}$ is finite dimensional.
 
Finally, that $U$ lowers filtration levels by one is immediate.
\end{proof}
 
 %%%%%%%%%%%%SECTION%%%%%%%%%%%%%%%%
\section{The definition of $\U_K(t)$}\label{sectionU}
 
For each $t \in [0,2]$ and for all   $s\in \R$, the set $(\ck, \calf_t)_s \subset \ck$ is a subcomplex.   Thus, we can make the following definition.
 
\begin{definition}  Let $\nu (\ck,\calf_t) =   \min\{ s \ |\  \text{Image\ }( H_0((\ck, \calf_t)_s) \to H_0(\ck)    \text{ is surjective} \}$.
\end{definition}

\begin{definition} $\U_K(t) =-2  \nu(\ck ,\calf_t)$.
\end{definition}

\subsection{Example}
 
Consider the knot $K = T(3,7)$ with $\ck$ as illustrated in Figure~\ref{fig37}.  The portion of the complex shown has homology $\F$, at  grading 0.

The  subcomplex $(\ck,\calf_t)_s$ is   generated by the bifiltered generators with Alexander and algebraic filtration levels  satisfying 
\begin{equation}\label{eqn0}
 Alex \le \frac{2}{t} s +(1 - \frac{2}{t})Alg.
\end{equation}

\vskip.05in
\noindent{\bf Observation} The lattice points which contain a filtered generator at filtration level $t$ all lie on a line of slope $$m = 1 - \frac{2}{t},$$  with lattice points parametrized by the pair $(Alg, Alex)$.
Alternatively, if a line of slope $m$ contains   distinct lattice points representing bifiltration levels of generators at the same   $\calf_t$ filtration level,  then $$t = \frac{2}{1-m}.$$\vskip.05in

 In the diagram for $T(3,7)$ shown in Figure~\ref{fig37}, the  illustrated  line in  the plane corresponds to  $t = \frac{4}{5}$ and $s = 2$.  Since the lower half-plane bounded by this line contains a generator of $H_0(\ck)$, while no half plane bounded by a  parallel line with smaller value of $s$ contains such a  generator, we have $\U_K(\frac{4}{5}) = -2(2) = -4$.
 
Continuing with $K = T(3,7)$, it is now clear that for $m < -2$ (that is, for $t < \frac{2}{3}$), the least $s$ for which  $(\ck,\calf_t)_s$  contains a generator of $H_0(\ck)$ corresponds to the line through $(0,6)$, which has filtration level $ \frac{t}{2} 6 + (1-\frac{t}{2})0 = 3t$.
 
For $-2 <m < -1$ (that is, for $\frac{2}{3} < t<1$), the least $s$ for which  $ (\ck,\calf_t)_s$ contains a generator of $H_0(\ck)$ corresponds to the line through $(2,2)$, which has filtration level  $\frac{t}{2} 2 + (1-\frac{t}{2})2 = 2$.  
Multiplying by $-2$  and checking the value $t = \frac{2}{3}$ yields 

$$\Upsilon_{T(3,7)}(t) = \begin{cases}
-6t & \mbox{if } 0\le  t \le \frac{2}{3}\\
-4 & \mbox{if } \frac{2}{3} \le t \le 1.
\end{cases}
$$

%%%%%%%%%%%%SECTION%%%%%%%%%%%%%%%%

\section{An alternative definition of $\nu$ and  $\U$}

In the appendix we prove Theorem~\ref{theorem:filteredbasis}, which has as an immediate consequence the following result.

\begin{theorem}\label{theorem:altdef}
The  filtered graded  chain complex $(\ck, \calf_t)$ is isomorphic to a filtered graded complex of the  form $$\calt \oplus \cala,$$ where $\calt \oplus \cala$ has the structure of a $\L$--module and the isomorphism is a $\L$--module isomorphism. The summand  $\calt$ has the properties that: (1)  it is isomorphic to $\L$ as a $\L$--module; (2)  the element $1 \in \L \cong \calt$ has grading 0.  Furthermore,  $\cala$ is acyclic as an unfiltered  complex.
\end{theorem}

Notice that since all gradings  in $\calt$ are even, the boundary operator restricted to $\calt$ is trivial.

When placed in this simple form, the computation of $\nu((\ck, \calf_t)) $ is simple: it is the $\calf_t$ filtration level of $1 \in \L \cong \calt$.  Hence, we have the following result.

\begin{corollary} $\U_K(t)$ equals $-2$ times the  $\calf_t$--filtration level of $1 \in \L \cong \calt$ for the decomposition $(\ck, \calf_t) \cong \calt \oplus \cala$.
\end{corollary}

%%%%%%%%%%%%SECTION%%%%%%%%%%%%%%%%

 \section{Products  and additivity}

According to~\cite{os-knotinvars}, there is a (graded) chain homotopy equivelance of complexes $$\crm(K_1) \otimes_\L \crm(K_2) \simeq \crm(K_1 \cs K_2)$$ that preserves the $\L$--structure.

Each of $\crm(K_1)$, $\crm(K_2)$ and $\crm(K_1 \cs K_2)$ has an algebraic filtration.   To distinguish these, we write $Alg^1$, $Alg^2$ and $Alg^{1,2}$.  Similarly, the Alexander and  $\calf_t$ filtrations will be  distinguish with  superscripts.

Momentarily we write $\crm_1 = \crm(K_1)$ and $\crm_2 =  \crm(K_2)$.  For each  $t \in [0,2]$ the filtrations $\calf^1_t$ and $\calf^2_t$ on $\crm_1 $ and $\crm_2$  induce a filtration   $\calf_t^1\otimes \calf_t^2$ on $\crm_1\otimes_{\L}  \crm_2$, defined via:
 \begin{multline*} (\crm_1 \otimes_{\L}  \crm_2,\calf_t^1 \otimes\calf^2_t )_s =\\ \text{Image} (\oplus_{s_1 +s_2 = s}\ (\crm_1, \calf^1_t)_{s_1} \otimes_\F (\crm_2, \calf^2_t)_{s_2} \to (\crm_1, \calf^1_t) \otimes_\L 
(\crm_2, \calf_t^2)).
\end{multline*}
  Notice that the direct sum is infinite and each summand is infinitely generated. Again, according to~\cite{os-knotinvars}, for the connected sum of knots, the equivalence $$\crm(K_1) \otimes_\L \crm(K_2) \simeq \crm(K_1 \cs K_2)$$ is a filtered equivalence  for   both the Alexander and algebraic filtrations.  To state this explicitly, 
  $$(\crm(K_1),Alex^1) \otimes_\L (\crm(K_2),Alex^2) \simeq ( \crm(K_1 \cs K_2), Alex^{1,2})$$
and $$(\crm(K_1),Alg^1) \otimes_\L (\crm(K_2),Alg^2) \simeq ( \crm(K_1 \cs K_2), Alg^{1,2}).$$
 
\begin{theorem} \label{theorem:product} For all $t \in [0,1]$,
 $$(\crm(K_1),\calf^1_t) \otimes_\L (\crm(K_2), \calf^2_t) \simeq ( \crm(K_1 \cs K_2),\calf^{1,2}_t).$$ \end{theorem}
 
 \begin{proof} Fix   bases $\{x_i\}$ and $\{y_i\}$  for the free $\L$--modules $\crm(K_1)$ and $\crm(K_2)$ so that  the sets of all translates $\{U^k x_i\}$ and $\{U^k y_i\}$, $k \in \Z$, form graded bifiltered bases for $\crm(K_1)$ and $\crm(K_2)$ (as $\F$--vector spaces).  The $\F$-vector space $\crm(K_1) \otimes_\L \crm(K_2)$ is generated by the set of all  tensor products, $\{U^k x_i \otimes U^j x_l\}$, but note that these do not form a basis; for instance, $U x \otimes y = x \otimes U y$.  
 
When selecting elements from  $\{U^k x_i\}$, we will sometimes refer to them as $x$; similarly for $y$.  Note that in particular, for such basis elements,    $Alg^{1,2}(x \otimes y) = Alg^1(x) + Alg^2(y)$ and $Alex^{1,2}(x \otimes y) = Alex^{1}(x) + Alex^{2}(y)$.
 
 The proof of the theorem consists of showing that the filtrations $\calf^1_t \otimes \calf^2_t$ and $\calf^{1,2}_t$ on $\crm(K_1) \otimes_\L \crm(K_2)$ are the same.  
 
 If an element  $z \in \crm(K_1) \otimes_\L \crm(K_2)$ has $\calf^{1,2}_t$ filtration level $s$, then it can be written as the sum of elements $x \otimes y$ with $$\frac{t}{2}Alex(x \otimes y) + (1-\frac{t}{2})Alg(x \otimes y) \le s.$$ This is the same as
 $$\frac{t}{2}Alex(x  ) + (1-\frac{t}{2})Alg(x) +  \frac{t}{2}Alex(y ) + (1-\frac{t}{2})Alg(y)\le s.$$  This implies that $\calf^1_t(x) + \calf^2_t(y) \le s$.  This in turn implies that $(\calf^1_t \otimes \calf^2_t) (x\otimes y) \le s$.  
 Thus, for all $z \in \crm(K_1) \otimes_\L \crm(K_2)$, $(\calf^1_t \otimes \calf^2_t)(z) \le \calf^{1,2}_t(z)$.
 
 Similarly, suppose that $z \in \crm(K_1) \otimes_\L \crm(K_2)$ has $\calf^1_t \otimes \calf_t^2$ filtration level $s$. Then it is the sum of elements $x\otimes y$, each of which satisfies $\calf_1^t(x) +\calf_2^t(y) \le s$.  This can be expanded and rewritten as $$\frac{t}{2}(Alex(x) + Alex(y)) +(1-\frac{t}{2})(Alg(x) + Alg(y)) \le s.$$  In other words, $z$ is the sum of elements $x\otimes y$ with $\calf^{1,2}_t (x\otimes y) \le s$.  Hence, $\calf^{1,2}_t(x \otimes y) \le s$.
 \end{proof}
 
Theorem~\ref{theorem:product}, along with Theorem~\ref{theorem:altdef}, offers a fast proof of the additivity of $\U$.

\begin{theorem} For each $t \in [0,2]$, $\U_{K_1 \cs K_2}(t) =  \U_{K_1 }(t) + \U_{  K_2}(t).$ 
\end{theorem}
  
\begin{proof} One only needs to check this for complexes of the form $\calt \oplus \cala$, as given in Theorem~\ref{theorem:altdef}.  Acyclic summands do not affect the value of $\U_K(t)$.  Thus, we only need consider the case of complexes $\calt(K_1) \otimes_\L \calt(K_2)$, for which the statement is clear. \end{proof}

 Similarly, Theorem~\ref{theorem:altdef} offers a fast proof of the following.
\begin{theorem} For an arbitrary knot $K$,  $\U_{-K}(t) = -\U_K(t)$.
\end{theorem}
\begin{proof} According to~\cite{os-knotinvars}, the complexes $\ck$ and $\text{CF}(-K)$ are duals: $\text{CF}(-K)\cong \ck^*$. More precisely, $\text{CF}(-K)$ is isomorphic to the complex  $\text{Hom}_{\F}(\ck, \F)$, having underlying vector space the space of $\F$--homomorphisms with finite dimensional (that is, finite) support.

If we fix a basis $\{x_i\}$ of $\ck$ as a $\L$--module so that the set $\{U^k x_i\}$ forms a graded bifiltered basis of $\ck$, then we can denote the elements of the dual basis   by $(U^kx_i)^*$.  The dual complex is readily understood in terms of these bases.
   
\begin{enumerate}
 \item An easy exercise shows that the  action of $U$ on the dual basis is of the form $U(U^kx_i)^* = (U^{k-1}x_i)^*$. In particular, the set $\{x_i^*\}$ forms a basis for the $\L$--module $\ck^*$.
 \item For any filtration $\calf$ on $\ck$, we can define a filtration $\calf^*$ on the dual space as follows:
$$(\ck^*, \calf^*)_{s} = \{\phi \in \ck^* \ |\  \phi( (\ck , \calf)_{-s'}) = 0 \text{ for all } s' > s\}.$$  The choice of signs ensures that the dual filtration is increasing.
Thus, $\calf^*(x_i^*) = - \calf(x_i)$.
\item The boundary operator for the dual space acts in the expected way with respect to basis elements:  if $x$ is a component of $\partial y$, then $y^*$ is a component of $\partial x^*$. 
\end{enumerate}

 These three observations are easily summarized in terms of  diagrams such as in Figure~\ref{fig37}: the diagram for $\crm(-K)$ is obtained from that for $\ck$ by rotating the figure by 180 degrees around the origin and reversing all the arrows.  
 
 There are two filtrations on $\crm(-K)$ of interest. The first is $\frac{t}{2}Alex^* + (1 -\frac{t}{2})Alg^*$; the second is $\calf_t^* =  (\frac{t}{2}Alex + (1 -\frac{t}{2})Alg) ^*$.  By using the chosen basis and its dual basis, it is possible to see that these  two filtrations are the same, as follows.   We use coordinates $(i,j)$ for the plane.   For a basis vector $x$, its  dual  vector $x^*$ is in $\calf_t^*$ if and only if it lies on or above the line $\frac{t}{2}j  + (1 -\frac{t}{2})i = -t$. If this is the case, then when rotated 180 degrees about the origin it lies on or below the line $\frac{t}{2}j  + (1 -\frac{t}{2})i = t$.   These are precisely the dual vectors for which $\frac{t}{2}Alex^* + (1 -\frac{t}{2})Alg^*  \le t$.
  
 The proof of the theorem is now reduced to an elementary calculation for the simple complex $\calt(K)$ and its dual $\calt(K)^*$. 
  \end{proof}
 
 %%%%%%%%%%%%SECTION%%%%%%%%%%%%%%%%
 \section{Basic properties of  $\U_K(t)$ and $\U'_K(t)$.}
We now present some basic results concerning  $\U_K(t)$  and its derivative.  An initial  observation is that $\U_K(0) = 0$ and, since $\ck$ is finitely generated, $\U_K(t)$ is continuous at 0. Thus, we focus on $t>0$.  

\begin{theorem}\label{mainthm1} $\ $\vskip.0in
\begin{enumerate}
\item  For every knot $K$, $\Upsilon(K)$ is a continuous piecewise linear function.   
\item At a nonsingular point of  $\U'_K(t)$, the value of  $|\U'_K(t)|$ is $|i-j|$, where $(i,j)$ is the bifiltration level of some filtered generator of $\ck$ with homological grading 0. \ 
 \item Singularities in $\U'_K(t)$ can occur only at values of $t$ such that some line of slope $1-\frac{2}{t}$ contains at least two lattice points, $(i,j)$ and $(i',j')$, each of which represents the algebraic and Alexander gradings of filtered generators of $\ck $ of homological  grading $0$.
 \item If $\U'_K(t)$ has a singularity at  $t$, then the jump in $\U'_K(t)$ at $t$, denoted $\Delta\U'_K(t)$, satisfies
 $| \Delta \U'_K(t) | = \frac{2}{t}| i - i'|$ for some   pair $(i, i') $ for which there are lattice points $(i,j)$ and $(i',j')$ as in the previous item.
 \end{enumerate}
 \end{theorem}
 \begin{proof}
 
 The proof is discussed in terms of the diagram of the complex, as illustrated for the knot $T(3,7)$ in the previous section.
 
 Suppose $\Upsilon_K(t) = -2s$ and there is precisely one lattice point $(i,j)$ with $\frac{t}{2}j + (1-\frac{t}{2})i = s$ which represents the bifiltration level of a filtered generator of $\ck$.   (This will be the case for all but a finite number of values of $t$.)  For a nearby $t$, say $t'$, the value of $\Upsilon_K(t') = -2s'$ will be such that the same vertex (at $(i,j)$) lies on the line $\frac{t'}{2}j + (1-\frac{t'}{2})i = s'$.  That is, for all nearby values of $t$, the value of $s$ is given by 
$\frac{t}{2}j + (1-\frac{t}{2})i$.  Written differently, $$\U_K(t) = -2i + (i-j)t.$$
 In particular, we see that $\U_K(t)$ is piecewise linear off a finite set.
 
Now consider a singular value of $t$, at which $\U_K(t) = -2s$ and there are two or more pairs $(i,j)$ for which $ \frac{t}{2}j + (1-\frac{t}{2})i = s$.  Notice that this line in the $(i,j)$--plane has slope $m= 1 - \frac{2}{t}$. For $t'$ close to $t$ and $t'<t$, we have  $$\U_K(t') =- 2i + (i-j)t' $$ for one of those pairs $(i,j)$.  If $t'$ is near $t$ and  $t' >t$,  then $$\U_K(t') = - 2i' + (i'-j')t' $$ for another (or possibly the same) of these pairs, $(i',j')$.
Notice that these are equal at  $t$, giving the continuity of $\U_K(t)$.
 
We now see that a singularity of $\U_K(t)$ occurs if $(j-i) \ne (j'-i')$.   With these observations, the proofs of (1), (2), and (3) are complete.
 
For (4),  our computations have shown that the change in $\U'_K(t)$, denoted $\Delta \U_K'(t)$, is given by $\Delta \U_K'(t) = (j-j') - (i-i')$
 for some appropriate $(i,j)$ and $(i',j')$.  Since both are assumed to lie on a line of slope $1 - \frac{2}{t}$, we have $j-j' = (1 - \frac{2}{t})(i-i')$, so 
 $$\Delta \U_K'(t) =  (1 - \frac{2}{t})(i-i') - (i-i') = -\frac{2}{t} (i - i').$$
 This completes the proof of the theorem.
 \end{proof}

\begin{corollary} For any knot $K$ and for $t = \frac{p}{q}$ with $\gcd(p,q) = 1$,      $$\frac{t}{2}  \Delta \U_K'(t)  = kp,  $$  where $k$ is some integer if $p$ is odd, or half-integer  if $p$   even.    
\end{corollary}
 
\begin{proof} By Theorem~\ref{mainthm1} (4),  $| \frac{t}{2}  \Delta \U_K'(t)  | =  |i - i'|$ for some pair of integers $i$ and $i'$, where there are two lattice points on a line of slope $m = 1 - \frac{2}{t}.$ Thus, we want to constrain the possible differences between the first coordinates of such lattice points.  
 
For $t=\frac{p}{q}$, $m  = -\frac{2q-p}{p}$.  Since $\gcd(p,q) = 1$, in reduced terms, this is either $m = -  \frac{2q-p}{p}$ or  $m = -\frac{q - (p/2)}{(p/2)}$ if $p$ is odd or even, respectively. Two lattice points on such a line have first coordinates differing by a multiple of $p$ or of $\frac{p}{2}$, if $p$ is odd or even, respectively.  The completes the proof.
 \end{proof} 

%%%%%%%%%%%%SECTION%%%%%%%%%%%%%%%%
 
 \section{The three-genus, $g_3(K)$.}\label{secgenus}

 \begin{theorem}   For nonsingular points of $\U'_K(t)$,  $|\Upsilon'_K(t)| \le g_3(K)$.   
  \end{theorem}
  
\begin{proof}   According to~\cite{os-GenusBounds}, if $K$ is of genus $g$, then all elements of $\ck$ have filtration level $(i,j)$, where $$-g \le i-j \le g.$$  It follows immediately from the second statement of  Theorem~\ref{mainthm1}  that  $|\Upsilon_K'(t)| \le g_3(K)$. 
\end{proof}
We also observe that the genus of $K$ constrains the possible points of singularity of $\U'_K(t)$.
 
\begin{theorem}\label{gbound} Suppose that $\Upsilon'_K(t)$ has a singularity at $t = \frac{p}{q}$, with $\gcd(p,q) = 1$.  Then:
\begin{itemize}
\item If $p$ is odd, $q\le g_3(K)  $. 
\item If $p$ is even, $q\le  2g_3(K) $.
\end{itemize}
\end{theorem}

\begin{proof}
 
Suppose that a line of slope  $m = -\frac{a}{b} $, where $0<b <a $ contains two distinct points of the form $(i,j)$ with $| i - j| \le g_3(K)$.  It follows quickly that the genus bound implies $$a \le 2g_3(K) - b.$$
 
To express this in terms of $t$, suppose  $t = \frac{p}{q}$ with $\gcd(p,q) = 1$.   Then $$m = 1 - \frac{2}{t} = -\frac{2q-p}{p}.$$ 

If $p$ is odd, then $\gcd(2q-p,p) = 1$.  If $p$ is even, say $p = 2k$, then $\gcd(2q-p, p) = \gcd(2q, p) = 2$ and $m = -\frac{q - k}{k}$, with $q $ and $k$ relatively prime.

In the first case, with $p$ odd, we have $2q-p  \le 2g_3(K) - p $, so $q \le g_3(K)$.

In the second case, with $p$ even, we have $q-k \le 2g_3(K) - k$, so $q \le 2g_3(K)$.
\end{proof}

%%%%%%%%%%%%SECTION%%%%%%%%%%%%%%%%
 
\section{$\U_K(t)$ as a knot  concordance invariant}
  
If  knots $K_1$ and $K_2$ are concordant, then there is an equality among  $d$--invariants: $d(S^3_N(K_1) , \spinc_m) = d(S^3_N(K_2) , \spinc_m) $ for all $N \in \Z$ and $m\in \Z$, $-\frac{N-1}{2} \le m \le  \frac{N-1}{2}$.   Here $S^3_N(K)$ denotes $N$ surgery on $K$, $d$ is the Heegaard Floer correction term, and $\spinc_m$ is a Spin$^c$ structure, with $m$  given by a specific enumeration of Spin$^c$ structures; all are described in~\cite{os-absolute}.   (In the case that $N$ is  odd, this   range of $m$  includes all possible Spin$^c$ structures.)

If $N$ is large, then $d(S^3_N(K_1) , \spinc_0) = D(K) + S(N)$, where $D(K)$ is the largest grading of a class $z$  in the homology of $\ck_{\{i\le 0, j\le 0\}}$ for which $U^k  z $ is nontrivial  for all $k >0$, and $S(N)$ is some rational function defined on the integers, independent of $K$.

In the case that $K$ is slice,  we see that the maximal grading $D(K) = D(u)$, where $u$ is the unknot.  This implies that for a slice knot $K$, $D(K) = 0$.  We have a nesting of complexes $$\ck_{\{i\le 0, j\le 0\}} \subset (\ck, \calf_t)_0.$$  Since $(0,0)$ is at $\calf_t$ filtration level $0$,  it follows that  $\nu(\ck,\calf_t) \le 0$; thus $\U_K(t) \ge 0$.

However, $-K$ is also slice, so $ -\U_K(t) \ge 0$.  It follows that $\U_K(t) = 0$.  An additive invariant of knots that vanishes on slice knots is a concordance invariant.

%%%%%%%%%%%%SECTION%%%%%%%%%%%%%%%%

\section{The concordance-genus}
The concordance-genus $g_c(K)$ of a knot $K$, defined in~\cite{liv-concordgen}, is the minimal genus among all knots concordant to $K$.  Since $\U_K(t)$ is a concordance invariant, the genus bounds in Section~\ref{secgenus} apply to the concordance genus.

\begin{theorem} For all nonsingular points of $\U_K(t)$, $|\U'_K(t)| \le g_c(K)$.  The jumps in $\U_K'(t)$ occur at rational numbers $\frac{p}{q}$.  For $p$ odd, $q \le g_c(K)$.  If $p$ is even, $\frac{q}{2}  \le  g_c(K)$.
\end{theorem}
 
 %%%%%%%%%%%%SECTION%%%%%%%%%%%%%%%%
\section{Bounds on the four-genus, $g_4(K)$.}\label{genus-bounds}
 
Let $\crm(K)_{0,m}$ denote the bifiltered subcomplex $\ck_{\{i\le 0 , j\le m\}}$.  We let $\nu^-(K)$ denote the minimum value of $m$ such that the homology of $\crm(K)_{0,m}$ contains a nontrivial grading 0 element of  the homology of $  \ck$, which we recall is  isomorphic to $\L$ with 1 at grading  0.  There is the following result of Hom and Wu~\cite{hom-wu}, built from work of Rasmussen~\cite{rasmussen}.  (In~\cite{hom-wu} the invariant $\nu^{+}$ is described; the equivalence with $\nu^{-}$ is presented in~\cite{oss}.)   
 
\begin{proposition}[Proposition 2.4, \cite{hom-wu}] $\nu^{-} \le g_4(K)$. 
\end{proposition}

Based on this, we show that $\U_K(t)$ provides a bound on $g_4(K)$.

\begin{theorem}For all $t \in [0,2]$, $   {|\U_K(t)|}   \le t g_4(K)$.
\end{theorem}
\begin{proof} Since $(0,m)$ is at $\calf_t$ filtration level $tm/2$, we have the containment $$\crm(K)_{0,m} \subset  (\crm(K), \calf_t)_{tm /2}.$$  Since $\crm(K)_{0,\nu^-}$ contains an element of grading 0 in the homology of $\ck$, so does the subcomplex  $(\ck, \calf_t)_{t\nu^-  /2}$.  Thus, $\nu(\ck, \calf_t) \le t \nu^{-} /2$.    By the previous proposition,   $\nu(\ck, \calf_t) \le t g_4(K)/2$.

Considering $-K$, we have  $\nu(\crm(-K), \calf_t) \le t g_4(-K)/2$; it follows that  $-\nu(\ck, \calf_t) \le t g_4(K)/2$. Combining these yields
$$|\nu(\ck, \calf_t)| \le t g_4(K)/2.$$
Multiplying by $-2$ yields the desired conclusion.
\end{proof}

%%%%%%%%%%%%SECTION%%%%%%%%%%%%%%%%
\section{Crossing change bounds}

Here we sketch a proof of Proposition 1.10 of~\cite{oss}.  The argument is essentially the same as used in~\cite{liv} to prove the corresponding fact about $\tau(K)$.

\begin{theorem} Let $K_-$ and $K_+$ be knots with identical diagrams, except at one crossing which is either negative or positive, respectively.  Then for $t\in [0,1]$,   $$\U_{K_+}(t)  \le    \U_{K_-} (t) \le  \U_{K_+} (t) +t.$$
\end{theorem}

\begin{proof}
   First note that $ K_- \cs -K_+$ can be changed into the slice knot $K_+ \cs -K_+$ by changing a negative crossing to positive.  Thus, $g_4(K_-  \cs-K_+) \le 1$.  It follows that 
\begin{equation}\label{eq1}
    -t \le \U_{K_-} (t)- \U_{K_+} (t)\le t.
\end{equation}

Next, note that  $ K_- \cs -K_+  \cs \ T(2,3)$ can be changed into the slice knot $K_+ \cs -K_+$ by changing one negative crossing to positive and one positive crossing to negative.  Thus, it too has four-genus at most 1: it bounds a singular disk with two singularities of opposite sign, and these can be tubed together.  A simple computation for $T(2,3)$ yields $\U_{T(2,3)}(t) = -t$ for $0\le t \le 1$.  Thus, 
 $$ -t \le \U_{K_-} (t)- \U_{K_+} (t) - t \le t,$$ which we rewrite as
   \begin{equation}\label{eq2}   0 \le \U_{K_-} (t)- \U_{K_+} (t)  \le 2t.   \end{equation}  Combining Equations~\ref{eq1} and~\ref{eq2},
   $$ 0 \le \U_{K_-} (t)- \U_{K_+} (t) \le t.$$ Adding $\U_{K_+}(t) $ to all terms yields the desired conclusion, 
     $$\U_{K_+}(t)  \le    \U_{K_-} (t) \le  \U_{K_+} (t) +t.$$
\end{proof}

\noindent{\bf Note} This argument can be easily modified to show that if there is a singular concordance from  $K$ to  $J$ with a single positive double point, then $\U_{K }(t)  \le    \U_{J} (t) \le  \U_{K} (t) +t.$ 

%%%%%%%%%%%%SECTION%%%%%%%%%%%%%%%%
\section{The Ozsv\'ath-Szab\'o $\tau$-invariant and $\U_K(t)$ for small $t$}

For small $t$, $\U_K(t)$ is determined by the  $\tau$ invariant defined in~\cite{os-four-genus}.  We review the definition below.  Here is the statement of the result.

\begin{theorem}\label{tauthm} For $t$ small, $\U_K(t) = -\tau(K) t$. 
\end{theorem}

The subquotient complex $\ck_{\{i\le 0\}} / \ck_{\{i<0\}}$ will be denoted $\widehat{\crm}(K)$.  (Usually, $\widehat{\crm} $ is written $\widehat{\cfk}$.) It is filtered by the Alexander filtration and has homology $\F$, supported in grading 0.  The invariant $\tau(K)$ is defined to be the least integer $\tau$ such that the map on homology $H_0(\widehat{\crm}(K)_{\{j\le \tau\}}) \to H_0(\widehat{\crm}(K))\cong \F$ is surjective.

We wish to relate $\tau(K) = \tau$ to an invariant of $\ck$.  The needed technical result is the following.  

\begin{lemma}\label{lemmatau} If $\tau(K) = \tau$, then there is a cycle $ w \in \ck_{\{i\le 0, j\le \tau \} \cup \{i<0 \}}$ representing a nontrivial element in $H_0(\ck)$.
\end{lemma} 

\begin{proof}
From the definition of $\tau$ we see that there is a chain $x\in \ck_{\{i\le 0, j\le \tau\} \cup \{i<0\} }$ that in the quotient $\widehat{\crm}(K)$ is a cycle that represents a generator of $H_0(\widehat{\crm}(K))$.   

Since the chain $x$ represents a cycle in  $\widehat{\crm}(K)$, it   has the property that $\partial x = y$, where $y \in \ck_{i<0}$.  Note that $y$ is a cycle and $gr(y) = -1$.  Since $H_{-1}(\ck_{i<0}) = 0$, there is a chain $z \in \ck_{i<0}$ with $\partial z = y$.  Thus, $x + z $ is  a cycle in $\ck_{\{i\le 0, j\le \tau\} \cup \{i<0\} }$.  The map $H_0(\ck_{i\le 0}) \to H_0(\widehat{\crm}(K))$ is an isomorphism; both groups are isomorphic to $\F$. Thus, $x + z$ represents a generator of $H_0(\ck_{i\le 0}$).
The map  $H_0(\ck_{i\le 0}) \to H_0(\ck)$ is an isomorphism, completing the proof.
\end{proof}

\begin{proof}[Proof, Theorem~\ref{tauthm}]
For $t$ small, we consider the filtration $\calf_t$  and   the filtration level $s= \frac{t}{2}\tau$.   Then one has $\ck_{s} = \ck_{\{i\le0,j\le \tau\} \cup \{i<0 \}}$.  By Lemma~\ref{lemmatau}, this subcomplex contains a cycle that represents an element of grading 0 in $H(\ck)$.  Thus, for this $\calf_t$ filtration, $\nu \le \frac{t}{2}\tau$.

On the other hand, suppose that  $\nu <\frac{t}{2}\tau$. Then there would exist a cycle    $$z \in \ck_{\{i\le0,j\le \tau-1 \}\cup \{i<0\}}$$ representing a generator of  $H(\ck)$ of grading 0.  However, the image of $z$ in $\widehat{\crm}(K)$ would be an element in $\widehat{\crm}(K)_{\tau -1}$ that represents a generator of $H_0(\widehat{\crm}(K))$.  But $\tau$ is by definition the lowest level at which this can occur.  Thus, we see that $\nu =\frac{t}{2} \tau$.

To conclude, recall that $\U_K(t) = -2\nu$, so $\U_K(t) = -\tau(K)t$, as desired.
\end{proof}

\noindent{\bf Note.}  With care, one can check that in this argument, the condition that $t$ be small can be made precise by requiring that $t < 1/g_3(K)$.  Of course, once the result is established for some set of small $t$, then Theorem~\ref{gbound} provides the bound $t<1/g_3(K)$.

%%%%%%%%%%%%SECTION%%%%%%%%%%%%%%%%
\section{Equivalence of definitions of $\U_K(t)$}
In this section we explain why $\U_K(t)$ as defined here agrees with that of~\cite{oss}.

Beginning with $\ck$, a new complex $t\crm(K)$ can be  constructed as follows.   As an $\F$--vector space, 
$$t\crm(K) =\ck \otimes_\L \F[v^{1/n}, v^{-1/n}],$$ where $U$ acts on $ \F[v^{1/n}]$ via multiplication by $v^2$. This has the structure of  an $  \F[v^{1/n}, v^{-1/n}]$--module.  To simplify notation, we write  $\L' =  \F[v^{1/n}, v^{-1/n}]$.

There are    (rational) filtrations  $Alg$ and $Alex$ on $t\ck$ which are consistent   with those on the  $\L$--submodule $\ck$.  The action of $v^{1/n}$ lowers filtration levels by ${1}/{2n}$.   Thus, $U = v^2$ lowers filtration levels by 1, as it should.  Similarly, the Maslov grading $M(x)$ naturally extends to $t\ck$ so that the action of $v^{1/n}$ lowers this grading by $1/n$, and thus $U = v^2$ continues to lower the Maslov grading by 2.
 
There is a rational grading on $t\ck$ defined via the Maslov grading, $M$,  along with the algebraic  and Alexander filtrations.  If $x$ is an element at filtration level $(i,j)$, then: 
 \begin{equation}\label{eqn1b}
 \text{gr}_t( x ) = M(x) - t(j-i).
 \end{equation} 
(In~\cite{oss}, only generators at algebraic filtration level 0 are used to define gr$_t$, so $i=0$  and the formula $\text{gr}_t( x ) = M(x) - tAlex(x)$ is presented.)  One checks that $U$  to lowers $\text{gr}_t$--gradings by 2, so    on the extension to $t\crm(K)$, $v$ lowers gradings by $1$ and $v^{1/n}$ lowers gradings by $1/n$.

If $x$ is a filtered generator of $\ck$   with $\partial x = \sum y_l$, then the boundary  $\partial_t$ is  defined so that $\partial_t x = \sum v^{\alpha_l}y_l \in t\crm(K)$, with the values of $\alpha_l$  given explicitly in~\cite{oss}.  This extends naturally to a boundary operator on all of $t\crm(K)$.

Given that the operator $\partial_t$ is well-defined, it is a simple matter to determine its value.  Suppose that $x$ is a filtered generator of $\ck$ at filtration level $(i,j)$,  Maslov grading $g$, and suppose also that $\partial x = \sum y_l$.  Let $y$ denote one of the terms in this sum, at filtration level $(i',j')$, necessarily of grading  $g-1$. Then viewed as an element of $t\crm$, $x$ is of grading $g - t(j-i)$, and $y$ has grading $g-1 - t(j'-i')$.  In $\partial_t x$, the term $ v^\alpha y$ appears, and $\alpha$ is such that  gr.$_t(v^\alpha y) = \text{gr}_t(x) -1$.   Rewriting this, we have $(g-1) - t(j'-i') - \alpha = g - t(j-i)-1$.  That is, 
\begin{equation}\label {eq3}
\alpha = t(( j - j') - (i-i')).\end{equation}

As two examples, Figure~\ref{fig37a} illustrates the complexes $t\crm(K)$ for $K = T(3,7)$, with $t= \frac{1}{3} $ and $t= 2$.   The construction is straightforward using Equation~\ref{eqn1b} and the fact that $v$ shifts along the diagonal a distance of $1/2$ down and to the left.  The portion of the complex illustrated was chosen because its homology is $\F$ in grading 0 and represents the  generator of the homology of $t\crm$ in grading 0.  In the case that $t=\frac{1}{3}$, the full complex consists of the illustrated complex along with all its translates a distance $\frac{k}{6}$, $k \in \Z$, along the diagonal.  In the case of $t=2$, the translates are those a distance $\frac{k}{2}$ along the diagonal.
 
\begin{figure}[h]
 \center{\fig{.22}{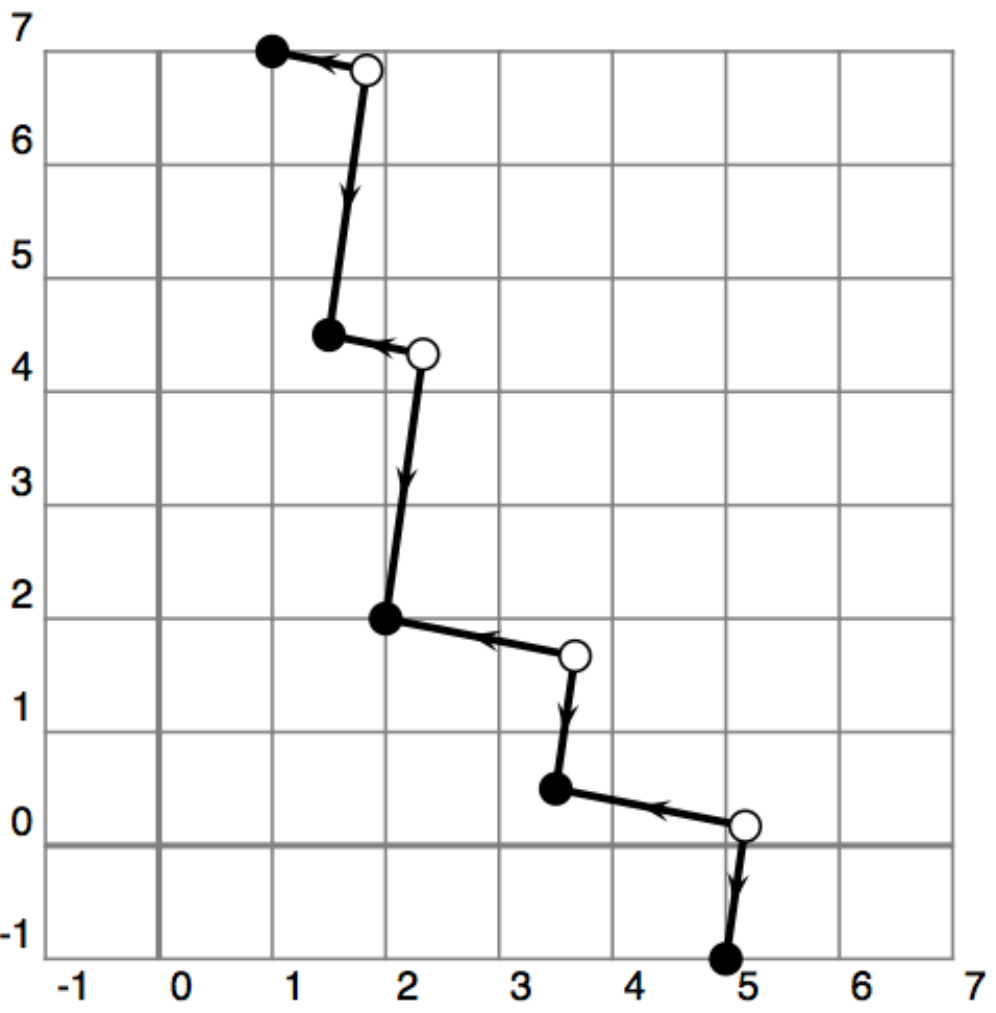}\hskip1in \fig{.22}{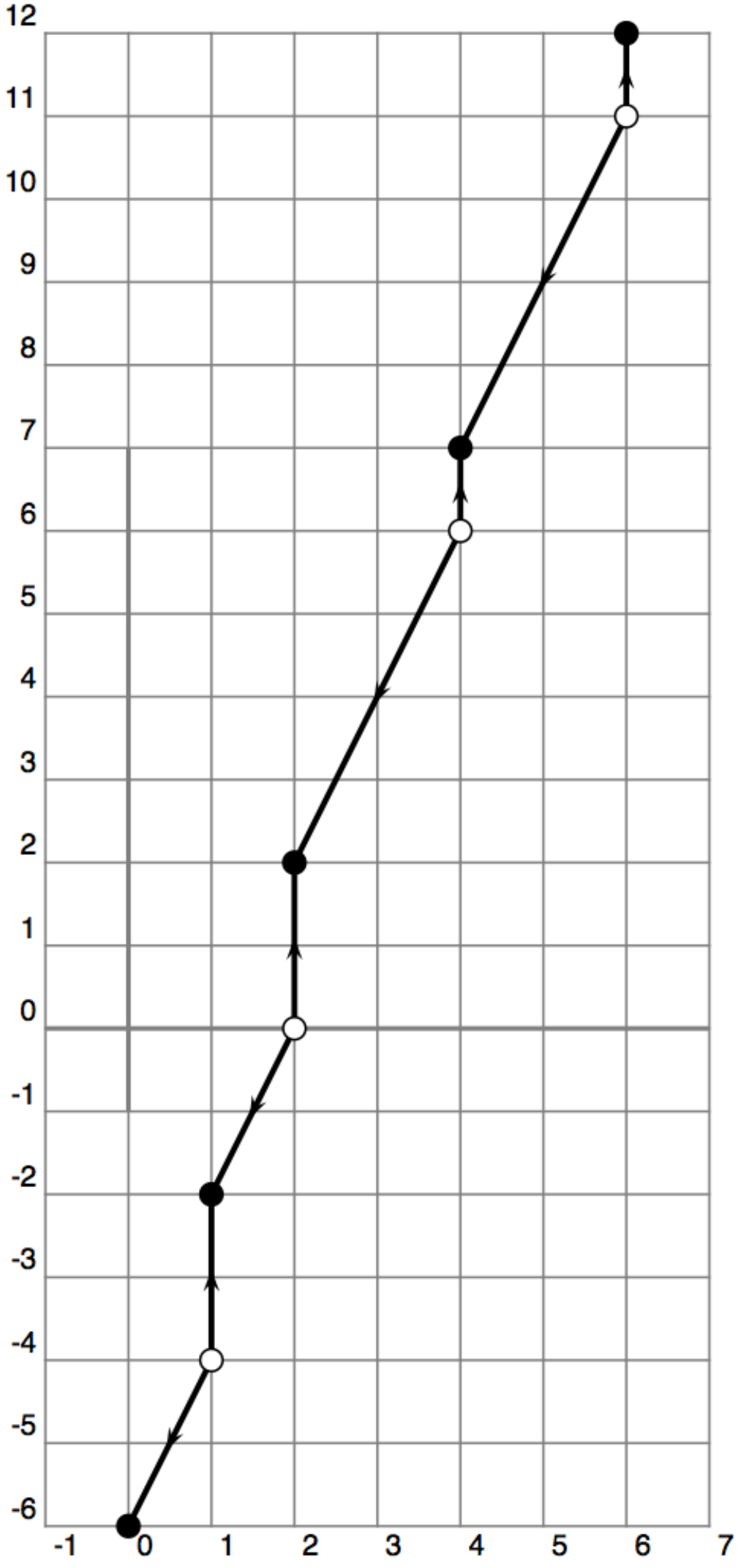}}
\caption{$t\crm_{t= 1/3} (T(3,7)) $ and $t\crm_{t= 2} (T(3,7)) $}
\label{fig37a}
\end{figure}

 It is apparent from these examples that the Alexander filtration is not a filtration of the chain complex, since some arrows increase the Alexander filtration level.  However, as is easily verified, the algebraic filtration is a filtration on the chain complex.
 
\begin{definition} For $t = \frac{m}{n}$, denote by $t\cfk^{-}(K)$ the complex $t\crm(K)_{{i\le 0}}$. \end{definition}
 
\noindent{\bf Note.} In~\cite{oss}, this complex is denoted $t\cfk(K)$.  In fact, it is the complex that is explicitly constructed.  Here we first introduced the infinity complex to be consistent with our earlier constructions.
 
\begin{definition} For $t = \frac{m}{n}$, $\U_K(t)$ is the maximal grading of a class in the homology of $t\cfk^{-}(K) $ that maps to a nontrivial element in the homology of $t\crm(K)$.  Equivalently, it is the maximal grading of a class in the homology of  $t\cfk^{-}(K) $  which is not in the kernel of $v^k$ for all $k >0$.
\end{definition}

\begin{lemma}The value of $\U_K(t)$ as just defined is equal to   $-2s$, where $s$ is the least number for which the homology of  $t\crm(K)_{ i\le s}$ contains an element of grading 0 that represents a nontrivial element of the homology of $t\crm(K)$.
\end{lemma}
\begin{proof} This follows from a simple change of coordinates. 
\end{proof}

\subsection{The two definitions of $\U_K(t)$ agree.}

Suppose that using this definition of $\U_K(t)$, we have  $\U_K(t) = -2s$.  This implies  that $t\crm(K)_{i\le s}$ contains a cycle $z$ representing a nontrivial generator of grading 0 in the homology of $t\crm(K)$.  Write $z = \sum x_l$, where the $x_l$ are filtered generators.  Some $x_l$ has filtration level $(s,j)$, and none of the $x_l$ has algebraic filtration level greater than $s$.  

From the regrading formula given in Equation~\ref{eqn1b}, $\text{gr}_t( x ) = M(x) - t(j-i) $,   we see that   generators of $\ck$ at filtration level $(i,j)$ and grading 0 yield generators of grading 0 in  $t\crm(K)$ at filtration level  $(i +\frac{t}{2} (j-i), j + \frac{t}{2}(j-i))$. (Recall that shifting down and to the left by $t$ units decreases the grading  by $2t$.)
We are thus led to consider the transformation $$(i,j) \mapsto (( 1 -\frac{t}{2})i+ \frac{t}{2}j, -\frac{t}{2}i + (1+\frac{t}{2})j).$$ 
Its inverse is given by 
$$(i,j) \mapsto (( 1 +\frac{t}{2})i- \frac{t}{2}j,  \frac{t}{2}i + (1-\frac{t}{2})j).$$ 
Under this transformation, for  a fixed value of $s$,  the vertical  line       $\{(s, z)\ | \ z \in {\mathbb{R}} \}$, is carried to the line (in the $\ck$--plane)  $\{ (( 1 +\frac{t}{2})s- \frac{t}{2}z,  \frac{t}{2}s + (1-\frac{t}{2})z)\ |\ z \in {\mathbb{R}} \}.$  Relabeling the coordinate system $(x,y)$, this is the line $$y = (1-\frac{2}{t} ) x + (\frac{2}{t})s.$$

Comparing with Equation~\ref{eqn0}, we   see that the homology of the  filtered complex $(\ck, \calf_t)_s$ contains a generator of grading 0 that is nontrivial in the homology of $\ck$, and that this is not the case for $(\ck, \calf_t)_{s'}$ for any $s' < s$. Thus, the value of $\U_K(t)$ as defined in Section~\ref{sectionU} is $-2s$, and the definitions agree.

%%%%%%%%APPENDIX%%%%%%%	

\appendix

\section{A structure theorem for $\ck$.}

In~\cite[Chapter 11]{lot}, vertical and horizontal reductions of $\ck$ are discussed.  That presentation   applies to the filtered complex $(\ck, \calf_t)$, but adjustments in the details would be required because, for instance, the horizontal and vertical filtrations are integer valued rather than being real filtrations.  Since the argument in the present case is  straightforward, we present it in detail.

  Viewed as a $\L$--module, $\ck$ is freely generated by a  finite set $\{w_i\}_{1\le i \le m}$. We again simplify notation by  suppressing the indexing set and write $\{w_i\}$. This set can be chosen so that the set $\{U^k w_i\}_{k \in \Z}$ forms a bifiltered graded basis for  the $\F$--complex $\ck$.  We will refer to any such set $\{w_i\}$ as a $\L$--basis for $\ck$.  A $\L$--module change of basis among the $w_i$  that preserves gradings and filtration levels  induces a change of bifiltered graded basis for the $\F$--complex $\ck$.  We will refer to any such change of basis as a $\L$--change of basis of  $\ck$.  Analogous notation will be used when working with the filtered graded complex $(\ck,\calf_t)$.  
  
  \begin{theorem} \label{theorem:filteredbasis} Let $t \in [0,2]$.  As a $\L$--module, $\ck$ has a basis $\{\alpha, \beta_1, \ldots , \beta_k\}$, inducing a splitting of $\ck$  (as a $\L$--module) as the direct sum   $\ck \cong  \calt \oplus \cala$, where $\calt$   is freely generated by $\alpha$ and $\cala$  is freely generated by $\{\beta_1, \ldots , \beta_k\}$.  This splitting has the following properties.
  
  \begin{itemize}
\item  $(\ck , \calf_t) \cong \calt \oplus \cala$ as a filtered graded $\F$-complex.
\item The complex $\calt$ has filtered graded basis $\{U^k \alpha\}_{k\in \Z}$, the boundary map is trivial on $\calt$, and $gr(\alpha) = 0$.
\item The complex $\cala$ has filtered graded basis $\{U^k \alpha_i\}_{k\in \Z}$ and has trivial homology: $H(\cala) = 0$.
\end{itemize}
\end{theorem}
   
\begin{proof}  We begin with the $\L$--generating set of $\ck$, $\{w_i\}$.  By replacing generators with their $U^k$ translates and renaming the generators,  we can decompose this  into two subsets: $\{x_i\}$, all of grading 0, and $\{y_i\}$, all of grading $1$.
   
To simplify notation, we abbreviate the filtered graded $\F$--complex $(\ck,  \calf_t)$ by $\ct$.
\begin{enumerate}
\item Let $A$ be a cycle in $\ct$  having the least filtration level among cycles representing nontrivial classes in $H_0(\ct).$  After reordering the generators, we can write $A =  x_1 + \cdots +  x_k$, with the filtration levels nonincreasing.  Replacing $x_1$ with   $x_1 + \cdots +  x_k$ as the first generating element (over $\L$) induces a filtered change of basis for $\ct$.  Thus, the first element of the $\L$--basis, which we now denote $A_1$, is a cycle of least filtration level representing a nontrivial element of $H_0(\ct)$.
\item Consider the set of all  generating elements $y_i$   that have the property that $A_1$ is a component of $\partial y_i$. After reordering the basis, we can assume these are $\{y_1, y_2, \ldots , y_k\}$ for some $k$, and that the filtrations are in nondecreasing order.  Make the $\L$--change of basis that replaces each $y_i$, $2 \le i \le k$, with $y_i + y_1$.  This  induces a  filtered change of basis of $\ct$.  Now, the only generator   having $A_1$ as a component of its boundary is $y_1$, which we relabel $B_1$. \vskip.05in
\item  After perhaps reordering the $x_i$, we  have either $\partial B_1 = A_1$ or $\partial B_1 = A_1 +x_2 + \cdots +x_k$  for some $k \ge 2$, with the filtration levels nonincreasing.    Since  $\partial^2 = 0$, it follows that $B_1$ is not a component of any element in the image of $\partial$.   
   
If $\partial B_1 = A_1$, then we see that $\{A_1 , B_1\}$ generates an acyclic {\it summand} of $\ct$,   and thus $A_1$ would not represent a nontrivial element in homology.
    
We have  $\partial B_1 = A_1 + x_2 + \cdots x_k$ for some $k \ge 2$. Make the $\L$--change of basis that replaces $x_2$ with  $ x_2 + \cdots x_k$, now calling this new element $A_2$.  Then $\partial B_1 = A_1 + A_2$.  Note that since $A_1$ is a cycle and $A_1 +A_2 = \partial B_1$ is a cycle, that $A_2$ is a cycle representing the same homology class as $A_1$.  Hence the filtration level of $A_2$ is  greater than or equal to  that of $A_1$.  \vskip.05in
   
\item  We now repeat the previous argument, making a change of basis so that the only basis elements with boundary that include $A_2$ as a component are $B_1$ and perhaps a second generator that we denote $B_2$.   \vskip.05in
 \item This step-by-step procedure must eventually stop, at which time there is constructed a   summand of the $\F$--complex $\ct$    $$D= A_1 \leftarrow B_1 \rightarrow A_2 \leftarrow B_2 \rightarrow A_3\leftarrow \cdots \rightarrow B_{k-1}\rightarrow A_k.$$
   Note that the process must end with an $A_k$; if it stopped with a $B_k$, the resulting complex would be acyclic and thus not contain a nontrivial element in homology.  This complex is a summand of the complex $\ct$.  Note that $\L D$ is a summand of a direct sum decomposition  of $\ct$, as a subcomplex and also as a submodule of the $\L$--module.   \vskip.05in
\item Since $A_1$ has the lowest filtration level among the $A_i$, we can replace each $A_i$ with $A_1 + A_i$ to form a new basis. The complex then splits in the following way:
  $$A_1\ \  \oplus\ \  \left[B_1 \rightarrow (A_1+ A_2)  \leftarrow B_2 \rightarrow (A_1 + A_3)\leftarrow \cdots  \rightarrow B_{k-1} \rightarrow (A_1 +A_k)\right].$$
     We let $\calt = \L A_1$.  It satisfies the required conditions of the theorem.  Since as a $\L$--module, $H(\calt) \cong H(\ct)$, the complementary summand to $\calt$ must be acyclic.  That complementary summand yields the summand $\cala$ in the statement of the theorem.     \vskip.05in
\end{enumerate}
\end{proof}

\end{document}